\documentclass{amsart}
\usepackage{amsfonts,amscd}

\newtheorem{theorem}{Theorem}[section]
\newtheorem{lemma}[theorem]{Lemma}
\newtheorem{corollary}[theorem]{Corollary}
\newtheorem{proposition}[theorem]{Proposition}
\theoremstyle{remark}

\theoremstyle{definition}
\newtheorem{definition}[theorem]{Definition}

\numberwithin{equation}{section}
\makeatother

\newcommand{\I}{1\!{\mathrm l}}

\DeclareMathOperator{\Cdb}{{\mathbb C}}

\DeclareMathOperator{\Tdb}{{\mathbb T}}

\begin{document}

\title[Outers for noncommutative $H^p$ revisited]{Outers for noncommutative $H^p$ revisited}

\date{\today}
\thanks{The first author was supported by a grant from the NSF.
The contributions of the second author is based upon research supported by the National Research Foundation. Any opinion, findings and conclusions or recommendations expressed in this material, are those of the authors, and therefore the NRF do not accept any liability in regard thereto.}

\author{David P. Blecher}
\address{Department of Mathematics, University of Houston, Houston, TX
77204-3008}
\email[David P. Blecher]{dblecher@math.uh.edu}
\author{Louis Labuschagne}
\address{Internal Box 209, School of Comp., Stat. \& Math. Sci., NWU, Pvt. Bag X6001, 2520 Potchefstroom, South
Africa} 
\email{louis.labuschagne@nwu.ac.za}

\begin{abstract}  We continue our study of outer elements of the noncommutative $H^p$ spaces associated with 
Arveson's subdiagonal algebras.  We extend our generalized inner-outer factorization theorem, and our 
characterization of outer  elements, 
to include the case of elements with zero determinant.  
In addition, we make  several 
further contributions to the theory of outers.  For example, we generalize the classical fact that outers in $H^p$ actually satisfy the stronger condition that there exist $a_n \in A$ with $h a_n \in {\rm Ball}(A)$ and $h a_n \to 1$ in $p$-norm.
\end{abstract}

\maketitle

\section{Introduction}

Operator theorists and operator
algebraists have found many important noncommutative analogues of the classical
`inner-outer factorization' in the Hardy spaces $H^p$.
 We recall two such classical results \cite{Hobk} of interest to us:
If $f \in L^1$
with $f \geq 0$, then $\int  \log f \,
> - \infty$ 
if and only if $f = |h|$ for an outer $h \in H^1$.
 We call this the {\em Riesz-Szeg\"{o} theorem} since it is often
attributed to one or the other of these authors.   Equivalently,
if $f \in L^1$ with  $\int  \log |f| \,
> - \infty$, then $f = uh$, where $u$ is  unimodular and $h$ is outer.   It follows easily that 
if $f \in H^1$ then $u \in H^\infty$, that is $u$ is {\em inner}.   Here $H^p$ are the classical 
Hardy spaces, of say the disk.  Secondly,
outer functions may be {\em defined} in terms of a simple equation involving
$\int  \log |f|$.   
In an earlier paper \cite{BL6}  we found versions of these theorems
appropriate to (finite maximal)
{\em subdiagonal algebras} $A$ in the sense of Arveson \cite{AIOA}, thus solving old open problems in the subject
(see the discussion in \cite[Lines 8-12, p.\ 1497]{PX} and \cite[p.\ 386]{MMS}).
Subdiagonal algebras  are defined in detail below. Suffice it for now to say that $A$ is a certain 
weak* closed subalgebra of a von Neumann algebra $M$, where $M$ is assumed to possess a {\em faithful
normal tracial state} $\tau$. (In the classical case 
where $A$ is $H^\infty$ of the disk,
$\tau$ is integration around the circle $\Tdb$, and  
$M = L^{\infty}(\Tdb)$.)    The noncommutative Hardy space 
$H^p$ is simply the closure of $A$ in 
the noncommutative $L^p$-space $L^p(M)$ associated with   $\tau$.   
The role of the expression  
$\int  \log |f|$ frequently occurring in the classical case (more specifically of the geometric mean $\exp\int  \log |f|$), is played by the {\em Fuglede-Kadison
determinant} $\Delta(f)$.  Formal definitions and background facts may be found later in this introduction (or in our survey 
\cite{BL5}); in the next few paragraphs the general reader might simply want to keep in mind the two cases of Hardy spaces of the disk,  and of the subalgebra $A$ of upper triangular $n \times n$ matrices in $M = M_n$ (with 
$\tau$ equal to the normalized trace of a matrix).

  An element $h \in H^p$ is said to be {\em outer} if $\I \in [h A]_p$, the closure 
of $hA$ in the  
$p$-norm (or in the weak* topology if $p = \infty$).  Strictly speaking, we should call these
`right outers in $H^p$', but for brevity we shorten the phrase since we will
not really consider `left outers' in this paper, 
except in Section 3 where we show that left outer is the same as right outer 
under a certain hypothesis.  If at some point we mean that $h$ is both left and right outer then we shall explicitly say that.
In \cite{BL6} we showed that left outer is the same as right outer  
for  {\em strongly outer} elements
 $h$ -- these are the outers that satisfy   $\Delta(h) > 0$.   Every outer is strongly outer if the von Neumann subalgebra ${\mathcal D} \overset{def}{=} A \cap A^*$  is
finite dimensional (as it is in the classical case),  by \cite[Theorem 4.4]{BL6}
and the remark following it.   In actuality most of the theorems 
in the relevant section of \cite{BL6}
are about strongly outer elements, and not much was said about 
outers with $\Delta(h) = 0$.  From the classical perspective on the disk this is not an issue there since in that setting  $\int \, \log |f| > - \infty$ 
is the only  interesting case.  In the present paper we prove   
variants of our Riesz-Szeg\"{o} theorem and establish a characterization of outers that 
are  valid even when $\Delta(h) = 0$.  Also, in addition to improving some other results from \cite{BL6}, we make  
further contributions to the theory of outers.  For example we generalize the classical fact (see \cite[Theorem 7.4]{Garnett}) that outers in $H^p$ actually satisfy the stronger condition that there exist $a_n \in A$ with
$h a_n \in {\rm Ball}(A)$ and $h a_n \to \I$ in $p$-norm.  We call outers satisfying this stronger condition \emph{uniform outers}; this is  an intermediate concept to strongly  outer elements and general outers.    
 In Section 3 we exhibit another hypothesis which implies many of the  very strong results hitherto only known for strongly outers \cite{BL6}.
In this case too  every outer is uniform outer.  In this same setting we also characterize outers in terms of 
limits of strongly outer elements.

There are  two factors at play which dictate the nature 
of our main characterizations
 in Section 4, which are illustrated nicely by considering the extreme case that $A$ is a 
von Neumann algebra
(which is allowable in the noncommutative $H^p$ theory, but does not occur in the classical theory).   In this case,
$h$ being outer is the same as $h$ viewed as an unbounded operator being one-to-one with dense range.
Indeed by the polar decomposition it is equivalent to     
$|h|$ being outer, 
which in the theory of noncommutative $L^p$ spaces of finite von Neumann algebras (see e.g.\ \cite{JS}), is equivalent to being invertible as an unbounded
 operator, or to the support projection $s(|h|) = \I$.  In the case $A = L^\infty([0,1])$, any strictly positive function in $L^p$ will
thus be outer.   But how does one detect when $s(|h|) = \I$?  The Fuglede-Kadison determinant is irrelevant 
here when it is zero, e.g.\ in the case $A = L^\infty([0,1])$ it cannot distinguish between the cases
$h : [0,1] \to [0,1]$ and  $h : [0,1] \to (0,1]$. 
 Here we adopt a two-fold strategy in addressing this challenge. Factor (1) in remedying this, is to work with 
a quantity that looks similar to one of the known formulae (the one in (\ref{ArvS}) below) for the Fuglede-Kadison 
determinant.  This has the effect of
making our theorems look less elegant than their strongly outer counterparts in \cite{BL6}, but this seems to be the 
price to pay for complete generality.
Factor (2): since 
the problems we are trying to address in these characterizations only occur when the von Neumann algebra ${\mathcal D}$  above is large
 or complicated, a natural remedy is to  sometimes explicitly refer to ${\mathcal D}$ in 
the statements of our main theorems, using  ${\mathcal D}$ or its projection lattice (note that 
`testing against projections' easily  detects when $s(|h|) = \I$), as a `filtration' in some sense.

To illustrate these two factors, 
we state one  variant of each of our two main theorems from Section 4.  Combining the  above two 
factors  (1) and (2), for each nonzero orthogonal projection $e \in \mathcal{D}$, we modify the known formula
 (\ref{ArvS}) below for the Fuglede-Kadison determinant to the following `local version'.   Define for each such projection $e$  
$$\delta^e(x) =  \inf \{ \Vert f (1 + a_0) \Vert^e_p
 : a_0 \in A_0 \} , \qquad x \in L^p(M), $$
where $A_0$ is the ideal ${\rm Ker}(\Phi)$ in $A$, where $\Phi : M \to {\mathcal D}$ is the canonical projection (see 
below), and $\Vert \cdot \Vert^e_p$ is the `$p$-seminorm' associated with the  
trace $\tau_e(x) = \frac{1}{\tau(e)} \tau(x)$ on $eMe$, that is,  
$\Vert x  \Vert^e_p = 
(\frac{1}{\tau(e)} \tau((ex^* xe)^\frac{p}{2}))^{\frac{1}{p}}$ (the $e$'s inside the power are unnecessary if 
$x$ is in $L^p(eMe)$, as opposed to $L^p(M)$).    One may then define a real valued function $\delta(x)$ on 
the set of nonzero orthogonal projections $e \in \mathcal{D}$ by
$\delta(x)(e) = \delta^e(x)$.    We  have:

\begin{theorem}\label{Louter0} {\rm (Generalized noncommutative Riesz-Szeg\"{o} theorem) } \   If $f\in L^p(M)$,
where  $1 \leq p < \infty$, then the following are equivalent:
\begin{enumerate}
\item [(i)] $f$ is of the form $f=uh$ for some outer $h\in H^p$ and a unitary $u\in M$.
\item [(ii)] $\delta(f) > 0$. 
\item [(iii)]  $f d \notin [fA_0]_p$ whenever $0 \neq d \in \mathcal{D}$.
\end{enumerate}
If these hold and $f \in H^p$, then necessarily $u$ is inner (a unitary lying in $A$).  On the other hand, if $f\in L^p(M)_+$ then $f = |h|$ for an outer $h\in H^p$ 
if and only if $\delta(f) > 0$. 
\end{theorem}

In the classical or more generally the antisymmetric case,  the infimum in condition (ii) of the theorem is just $\Delta(f)$ by (\ref{ArvS}), so that (ii) is simply saying that $\Delta(f) > 0$.
Hence Theorem \ref{Louter0} contains the classical Riesz-Szeg\"{o} theorem, and its
extension to antisymmetric subdiagonal algebras.      
 
Before we state our second main theorem,  we prolong our discussion of factor (2) above, 
concerning problems stemming from the von Neumann algebra ${\mathcal D}$ introduced above, which is trivial in the classical case.
 It is obvious that if $h$ is outer, then so also is $\Phi(h)$ in $L^p({\mathcal D})$ (see e.g.\ line 4 of 
\cite[p.\ 6139]{BL6}). 
  Bearing all this  
in mind, it seems reasonable to characterize $h \in H^p$ being outer in
terms of $\Phi(h)$ being outer in $L^p({\mathcal D})$ plus one other condition yet to be specified 
(by what we said in Factor (1)\ above, $\Phi(h)$ outer happens 
if and only if $\Phi(h)$ or $|\Phi(h)|$ is
invertible as an unbounded operator). 
In fact in Lemma \ref{dlem} we show that for $h \in H^p$, if $\Phi(h)$ is outer then 
$h = uk$ where $k$ is an outer in $H^p$ and $u$ is a unitary which lies in $A$. 
Thus the `other condition' mentioned a line or two above, simply needs to force
this unitary $u$ to lie in ${\mathcal D}$ (because in that case it is easy to see that 
$uk$ is outer in $H^p$).     Here is such a formulation, in which the `other condition' involves a quantity similar to one of the known expressions
for $\Delta(h)$:

\begin{theorem}\label{Houterp0}  If $h \in H^p$, where $1 \leq p < \infty$, then $h$ is outer 
if and only if 
$\Phi(h)$ is outer in $L^p({\mathcal D})$, and $\Vert \Phi(h) \Vert_p = \delta^1(h)$ in the notation just   above Theorem {\rm \ref{Louter0}}.
\end{theorem}

We next review some of the definitions and notation underlying the 
 theory of noncommutative $H^p$-spaces. For a fuller account we refer the reader to our survey paper
\cite{BL5}.

Throughout $M$ is  a von Neumann algebra possessing a faithful
normal tracial state $\tau$.
Such von Neumann algebras are  
{\em finite} von Neumann algebras, and for such algebras
if $ab = 1$ in $M$, then $b a = 1$ also. Indeed, for any $a, b
\in M$, $ab$ will be invertible precisely when $a$ and $b$ are
separately invertible. 
In this framework the role of $H^\infty$ will be played by a (finite maximal) 
{\em subdiagonal algebra} in $M$.  The latter 
is defined to be a weak* closed unital 
subalgebra $A$ of $M$ such that if $\Phi$ is the unique trace-preserving conditional expectation from $M$ onto $A \cap A^* \overset{def}{=} {\mathcal D}$ 
guaranteed by \cite[p.\ 332]{Tak}, then:
$$ \Phi(a_1 a_2) = \Phi(a_1) \,  \Phi(a_2) , \; \; \; a_1, a_2 \in A . $$ (In the interesting {\em antisymmetric}
case, that is,  ${\mathcal D} = \Cdb 1$, we have  $\Phi(\cdot) = \tau(\cdot)\I$.)
In addition to this $A+A^*$ is required to be weak* dense in $M$.
Within the context of the first part of the definition, this last condition is 
actually equivalent 
to a whole host of alternative, but equivalent conditions 
(see e.g.\ \cite{BL5}). 

Of course $M$ itself is a maximal subdiagonal subalgebra (take $\Phi = Id$). 
The upper triangular matrices example mentioned earlier is the simplest non-commutative example of a maximal subdiagonal subalgebra; here  $\Phi$ is the 
expectation
onto the main diagonal. There are many more interesting examples incorporating 
 important constructs from the theory of operator algebras such as
 free group von Neumann algebras, the Tomita-Takesaki theory, etc.\  (see e.g.\ \cite{AIOA,Zs,MMS}).   

For a finite von Neumann algebra $M$ acting on a Hilbert space $H$, the set of possibly unbounded 
closed and densely defined operators on $H$ which are affiliated to $M$, form 
a topological $*$-algebra where the topology is the (noncommutative) topology 
of convergence in measure (see
Theorem I.28 and the example following it in \cite{Terp}). We will denote this 
algebra by $\widetilde{M}$; 
it is the closure of $M$ in the 
topology just mentioned. The trace
$\tau$ extends naturally to the positive operators in
$\widetilde{M}$. The important fact regarding this algebra, is that it is large 
enough to accommodate all the noncommutative $L^p$-spaces corresponding to 
$M$. Specifically if $1 \leq p < \infty$, then we define the space $L^p(M,\tau)$ by 
$L^p(M,\tau) =  \{a \in \widetilde{M} :  \tau(|a|^p) < \infty\}$,
where the ambient norm is given $\|\cdot\|_p = \tau(|\cdot|^p)^{1/p}$. The space 
$L^\infty(M,\tau)$ is defined to be $M$ itself. These 
spaces capture all the usual properties of $L^p$-spaces, with the dual action of $L^q$ on $L^p$ ($q$ conjugate to $p$) given by $(a,b) \to \tau(ab)$.  (See
e.g. \cite{Nel,FK,Terp,PX}).  For the sake of convenience we abbreviate $L^p(M,\tau)$ to
$L^p(M)$.
For any subset ${\mathcal S}$ of $M$, we write $[{\mathcal S}]_p$ for the
norm-closure of  ${\mathcal S}$ in $L^p(M)$, with the understanding that $[{\mathcal S}]_p$ 
will denote the weak* closure in the case $p=\infty$.  For $0 < p < \infty$ we define $H^p(M)$ and 
$H^p_0(M)$ to respectively  be $[A]_p$ and $[A_0]_p$. In the case 
$p=\infty$, we have $H^\infty(M)=A$ and $H_0^\infty(M)=A_0$. Whereas we will have 
occasion to consider $L^p$-spaces associated with various von Neumann algebras, 
the $H^p$-spaces in view in this paper will consistently be those defined above. 
For this reason we compress the notation $H^p(M)$ and $H^p_0(M)$ to $H^p$ and 
$H^p_0$ respectively.    The map $\Phi$ above extends to a contractive projection from $L^p(M)$ onto 
$L^p({\mathcal D})$, which we still write as $\Phi$.  

A crucial tool in the theory of noncommutative $H^p$ spaces 
is the {\em Fuglede-Kadison
determinant}, originally defined on $M$ by  $\Delta(a) = \exp
\tau(\log |a|)$ if $|a| > 0$, and otherwise, $\Delta(a) = \inf \,
\Delta(|a| + \epsilon \I)$, the infimum taken over all scalars
$\epsilon > 0$ (see \cite{FKa,AIOA}). The definition of this determinant 
extends to a very large class of elements of $\widetilde{M}$; large enough to 
accommodate all the $L^p$-spaces (see \cite{HS} for details). Arveson's Szeg\"o formula for elements of $L^1(M)_+$ is:
$$\Delta(h) = \inf \{ \tau(h |a+d|^2)
: a \in A_0 , d \in {\mathcal D} , \Delta(d) \geq 1 \}$$More generally for any $0<p,q<\infty$ and $h\in L^p(M)$ we have that (\cite{BL6}):
$$\Delta(h) = \inf \{ \tau(||h|^{p/q}(a+d)|^q)^{1/p}
: a \in A_0 , d \in {\mathcal D} , \Delta(d) \geq 1 \},$$
which has as a special case 
\begin{equation} \label{ArvS} \Delta(h) = \inf \{ \tau(|h (a+d)|^p)^{1/p}
: a \in A_0 , d \in {\mathcal D} , \Delta(d) \geq 1 \}. \end{equation} 

An element $h \in H^p$ is said to be {\em outer} if $\I \in [h A]_p$, or equivalently if $H^p = [h A]_p$.
By \cite[Lemma 4.1]{BL6} this is equivalent to be being in $L^p(M)$ {\em and} outer in $H^1$.
The noncommutative theory
of outers was initiated in \cite{BL3} and continued in \cite{BL6}. Subsequently the 
Riesz-Szeg\"o theorem and the theory of strongly  outer elements from the latter paper, was applied in \cite{LX} to
first prove a noncommutative version of the Helson-Szeg\"o theorem, after which this
theorem was used to give characterisation of invertible Toeplitz operators with noncommuting
symbols which is very faithful to the famous classical characterisation of Devinatz \cite{Dev}.

\section{Generalized invertibles in $H^p$} 

\begin{definition} Let $H^0$ be the closure of $A$ in the topology of convergence in 
measure. We say that $h\in H^0$ is outer in $H^0$ if and only if $hA$ is dense in $H^0$ 
with respect to the topology of convergence in measure.
\end{definition} 

Please note that the space $H^0$ as defined above is introduced for the sake of technical expediency rather than any attempt to 
identify an enveloping space of analytic elements. In fact even in the classical case of $H^p$ spaces of the disk, 
$H^0$ as defined above will be all of $\widetilde{M}$. To see this we note that by
a degenerate case of Mergelyan's theorem, $\bar{z}$ is uniformly approximable by polynomials in $z$ on the
circle minus an arc of length epsilon around 1.  So $\bar{z}$ is in the almost uniform closure
of $H^\infty(\mathbb{D})$
 (we thank Erik Lundberg for supplying this argument), and hence in $H^0(\mathbb{D})$.  So  $H^\infty(\mathbb{D}) + \overline{H^\infty(\mathbb{D})}$ is in $H^0(\mathbb{D})$. However $H^2(\mathbb{D})$ embeds continuously into $H^0(\mathbb{D})$. It is therefore obvious from the above that $L^\infty(\mathbb{T}) \subset L^2(\mathbb{T}) = H^2(\mathbb{D})+\overline{H^2(\mathbb{D})} \subset H^0(\mathbb{D})$, and hence that $\widetilde{L^\infty(\mathbb{T})} = H^0(\mathbb{D})$.

For elements of $H^p$ the concept of outerness may be regarded as a weak form of ``analytic'' invertibility. In our first two results we use the space $H^0$ to clarify this statement to some extent.

\begin{theorem} \label{isout}  $h\in H^0$ is outer in $H^0$ if and only if $h^{-1}$ exists as an element of $H^0$.
\end{theorem}

\begin{proof}
Firstly let $h$ be outer. Then there exists a sequence $\{a_n\}\subset H^0$ so that $ha_n\to \I$ in measure. Let $p$ be the left support projection of $h$. Then $p^\perp (ha^n)$ converges to $p^\perp$. Since by construction $p^\perp (ha^n)=0$ for all $n$, we must have  
$p^\perp=0$. So $h$ must have dense range. 

Now let $v$ be the partial isometry in the polar decomposition $h=v|h|$. We have just seen that $vv^*=p=\I$. Since $M$ is finite, this means that $v$ must be a unitary, and hence that $v^*v$.  Thus $|h|$ has dense range, and so it is 
one-to-one. Hence $h$ must be injective as well. Thus the formal inverse $h^{-1}$ of $h$ exists, and this is closed and densely defined by the basic theory of unbounded
operators, and 
it is even affiliated to $M$. Since $\{ha_n\}\subset H^0$ and since by continuity $h^{-1}(ha_n)$ will then converge in measure 
to $h^{-1}$, we have $h^{-1} \in H^0$.

Conversely, suppose $h^{-1}$ exists as an element of $H^0$. 
If $a_n \to h^{-1}$ in measure then $h a_n b \to hh^{-1} b= b$ in measure for any 
$b \in A$.  It follows that the closure in measure of $hA$ contains
$H^0$, the closure in measure of $A$.
\end{proof}

\begin{proposition}\label{isout2} Let $h\in H^p$ ($1\leq p\leq\infty$). If $1\leq p <\infty$ (respectively $p=\infty$) then $h$ is outer in $H^0$ in the sense above
if $hA$ is norm-dense in $H^p$ (resp.\  $hA$ is weak*-dense in $A$).
\end{proposition}

\begin{proof} First let $hA$ be norm-dense in $H^p$ where ($1\leq p <\infty$). Since the injection $L^p(M)\to \widetilde{M}$ is continuous, the norm density of $hA$ in $H^p$ implies that $hA$ is dense in the closure in measure of $H^p$ inside $H^0$. Since $A\subset H^p$, it is obvious that this must be all of $H^0$. 

Next suppose that $h\in A$ and that $hA$ is weak*-dense in $A$. Since the injection of $L^\infty$ into $L^1$ is continuous with respect to the weak topology on $L^1$, it is an exercise to see that then $hA$ is weakly dense in $H^1$. Since the weak and norm closures of $hA$ in $L^1$ are the same, this case therefore reduces to the one treated above. 
\end{proof}

We next present two fairly direct alternative descriptions of outerness, before proceeding with a more detailed analysis of this phenomenon. 

\begin{proposition}
Let $1\leq p \leq \infty$. Then $h\in H^p$ is outer if and only if $\Phi(h)$ is outer in $L^p(\mathcal{D})$ and $[hA_0]_p=H^p_0$.
\end{proposition}

\begin{proof} The ``only if'' part is known (see e.g.\ the proof of 
Theorem \ref{Houterp}). Hence suppose that $\Phi(h)$ is outer in $L^p(\mathcal{D})$ and $[hA_0]_p=H^p_0$. Select $\{d_n\}\subset L^p(\mathcal{D})$ so that $\Phi(h)d_n\to \I$ as $n\to\infty$. Now observe that $\{(h-\Phi(h))d_n\}\subset H^p_0$. Hence we may select $\{a_n\}\subset A_0$ so that $\|ha_n + (h-\Phi(h))d_n\|_p\to 0$ as
$n\to\infty$. But then $$\|h(d_n+a_n)-\I\|_p \leq \|\Phi(h)d_n-\I\|_p+\|ha_n + (h-\Phi(h))d_n\|_p\to 0 \quad\mbox{as}\quad n\to\infty.$$
The case $p = \infty$ is similar.   \end{proof}

\begin{proposition} \label{innam} Let $1 \leq p\leq \infty$. Then $h\in H^p$ is outer if and only if $\Phi(h)$ is outer in $L^p(\mathcal{D})$ and $\Phi(h)-h \in [hA_0]_p$.  \end{proposition}

\begin{proof} Since $\Phi(h)-h \in H^p_0$, the ``only if'' part follows from the previous proposition. Hence suppose that $\Phi(h)$ is outer in $L^p(\mathcal{D})$ and $\Phi(h)-h \in [hA_0]_p$.  
Since $1 \in [\Phi(h) A]_p$, it suffices to show that 
$\Phi(h) \in [hA]_p$.  However $\Phi(h) = (\Phi(h)-h) + h \in [hA]_p$.
\end{proof}

{\bf Remark.}  One may switch the condition $\Phi(h)-h  \in [hA_0]_p$ in the last result with $\Phi(h) \in [hA]_p$.

\begin{lemma} \label{gh}  Suppose that $(a_n)$ is a uniformly bounded sequence in $M$  in the operator  norm, and that 
$1 \leq p < \infty$.  The following are equivalent:
\begin{itemize}
\item [(i)] $a_n \to 1$ 
in measure, 
\item [(ii)] $a_n \to 1$  in $p$-norm. \end{itemize}   These imply
\begin{itemize} \item [(iii)] $a_n \to 1$  weak* in $M$,
 \end{itemize}  and all three are equivalent if the $a_n$'s are contractions.
\end{lemma}   

\begin{proof}  
If  $a_n \to 1$ 
in measure, then  $a_n \to 1$ in $p$-norm for any $p \in (0,\infty)$, by  \cite[Theorem 3.6]{FK}.
If  $a_n \to 1$ 
in $p$-norm for any $p \in (0,\infty)$, then  $a_n \to 1$ in measure, by  \cite[Theorem 3.7]{FK}.   Also,
since this is a uniformly bounded net in $M$ in the operator 
norm, it follows easily that $a_n \to 1$ weak* in $M$.   

If  $a_n \to 1$ weak* in $M$ then we claim that $a_n \to 1$ in 
measure whenever all the $a_n$'s are contractive elements of $M$.  Note that  $\tau(a_n)  \to 1$ in this
case, and $a_n \to 1$  weakly in
the $2$-norm.   We claim that $||a_n||_2 \to 1$, for if this was false there would be a subsequence
$||a_{n_k}||_2$ with limit $< 1$.  We may suppose that $||a_{n_k}||_2 \leq t < 1$.
Now $|\tau(a_{n_k})|  \leq ||a_{n_k}||_2 < t$, yielding in the limit with $k$ the contradiction $1 < t$.
Thus $a_n \to 1$ in measure by \cite[Theorem 3.7]{FK}  (iii).  
\end{proof}

We say that an element $h \in H^p$ is {\em uniform outer} in $H^p$, if there exists 
 a sequence  $a_n \in A$ such that 
$\{h a_n\}$ is a uniformly bounded sequence in $A$  in OPERATOR norm (that is in the $L^\infty$-norm on $M$), and $h a_n \to 1$ in measure
(or equivalently, by Lemma \ref{gh}, in $p$-norm).  This implies that  $h a_n \to 1$ weak*, 
and indeed it converges strongly too, and even $\sigma$-strong* (since it is uniformly bounded
and for example $\Vert (h a_n -1) x \Vert_2 \leq \Vert x \Vert \Vert h a_n -1 \Vert_2$ for any $x \in M$).  Note that in this case  
$h a_n \in H^p\cap M \subset H^1\cap M = A$ by \cite[Proposition 2]{Sai}.

\medskip

{\bf Remark.}  One may similarly define $h \in H^0$  to be {\em uniform outer in} $H^0$
if there is a sequence  $a_n \in A$ such that
$\{h a_n\}$ is a uniformly bounded sequence in $A$  in operator norm, which converges
to $1$ in measure (or equivalently, in any $p$-norm).

\medskip

In the classical case of the disk, if $f \in H^p$ it is known \cite[Theorem 7.4]{Garnett}
that being
`uniform outer' in $H^p$ is the same as being outer in $H^p$, and in fact one may ensure 
that $||f a_n||_\infty \leq 1$ above, and 
$f a_n \to 1$ a.e. We see next that  if $\Delta(h) > 0$, then $h$ is outer in $H^p$ if and only if $h$ is  uniform outer in $H^p$.

\begin{lemma} \label{ghg}  Let $1\leq p \leq \infty$.    If $h$ is uniform outer in $H^p$ then  $h$ is outer in $H^p$.
\end{lemma}   \begin{proof}   We know that $h a_n \to 1$ in $p$-norm, and weak*.  So this  is obvious.   
\end{proof}

\begin{theorem} \label{fak}  Let $1\leq p \leq \infty$.   Suppose that 
$h$ is outer in $H^p$ and $\Delta(h) \neq 0$ (the latter is automatic if $\mathcal{D}$ 
is finite-dimensional).  Then $h$ is uniform outer in $H^p$.
Indeed there exists outers  $a_n \in A$ with 
$\|h a_n \|_\infty \leq 1$ (this
is the operator norm), $a_n^{-1} \in H^p$, $\{\Vert  a_n^{-1}  \Vert_p\}$ bounded, 
$a_n^{-1} \to h$ in $p$-norm if $p < \infty$ (weak* if $p = \infty$), and
$h a_n \to 1$ weak* (or equivalently in $p$-norm, or in measure).
\end{theorem}  \begin{proof}  
 We follow the idea in the proof of \cite[Theorem 7.4]{Garnett},
but we will need several other ingredients, such as the functional calculus for unbounded
positive operators \cite[Theorem 5.6.26]{KR}, and some facts about the Fuglede-Kadison determinant of a positive unbounded operator that may be found summarized in for example the first pages 
of \cite{HS}.  In particular we use the spectral 
distribution measure $\mu_{|h|}$ on $[0,\infty)$ for $|h|$, namely
$\mu_{|h|}(B) = \tau(E_{|h|}(B))$ for Borel sets $B$, where 
$E_{|h|}$ is the projection valued spectral measure for $|h|$
 (see e.g.\ \cite{HS}).  
We have $\Delta(h) = \Delta(|h|) =
\exp \int_0^\infty \, \log \, t \, d \mu_{|h|}(t)$, and this is
strictly positive by hypothesis.          
Since $$\int_0^\infty \, \chi_B  \, d \mu_{|h|}(t) =
\mu_{|h|}(B) = \tau(E_{|h|}(B)) = \tau(\chi_B(|h|))$$ for Borel sets $B$, we have
\begin{equation}\label{specdis}
\int_0^\infty \, f \, d \mu_{|h|}(t) =
\tau(f(|h|))
\end{equation}
for simple Borel functions $f$ on $[0,\infty)$. Hence by
the functional calculus and Lebesgue's monotone convergence theorem,
the last equation also holds for every bounded
real-valued Borel function $f$ on $[0,\infty)$, since there
is an increasing sequence of simple Borel functions converging to $f$.  

Let $b_n(t) = 1/t$ if              
 $t > 1/n$, and  $b_n(t) = n$ if $0 \leq t \leq 1/n$.  Let $c_n(t) = t b_n(t)$.
  Then $c_n(t) = 1$ if $t > 1/n$, and otherwise equals
$t n$, which is majorised by 1.   
Let $k_n = b_n(|h|)$.  Then $k_n \in M$ by
the Borel functional calculus for unbounded operators.  

Now consider the operator $\frac{1}{b_n}(|h|)$. By the Borel functional calculus we then have that 
$$k_n\frac{1}{b_n}(|h|)=\mathrm{supp}(|h|).$$Since the injectivity of $|h|$ (see Theorem \ref{isout2}) combined 
with the Borel functional calculus for affiliated operators ensures that the spectral projection $E_{|h|}(\{ 0 \}) = 0$, it actually 
follows that $\frac{1}{b_n}(|h|)k_n=\I$ and hence that $k_n^{-1}=\frac{1}{b_n}(|h|)$ as affiliated operators. Since $\frac{1}{b_n}(t) \leq 1+t$, the Borel functional calculus also ensures that $$k_n^{-1} = \frac{1}{b_n}(|h|)\leq 
\frac{1}{n} \I+  |h|  \leq \I+ |h|.$$In view of the fact that 
$|h|\in L^{p}(M)$, this in turn ensures that $k_n^{-1}\in L^{p}(M)$. But then the determinant is well-defined for both $k_n$ and 
$k_n^{-1}$ with $1=\Delta(\I)=\Delta(k_n).\Delta(k_n^{-1})$. This can only be the case if $\Delta(k_n)>0$.

So by the noncommutative
Riesz-Sz\"ego theorem \cite{BL6} (applied to $A^*$), there exists a unitary 
$u_n \in M$ and an outer $a_n \in A$ with $k_n = a_n u_n = |(a_n)^*|$.
By the functional calculus,   
$$\Vert h k_n \Vert = \Vert |h| k_n \Vert =  \Vert c_n(|h|) \Vert \leq 1, $$ 
and so $\Vert h a_n \Vert \leq 1$.  Note that  
$ha_n$ is in $H^p \cap M = A$, and is outer in $H^p$. Hence it is 
outer in $A$.  

We show that $a_n^{-1}\in H^p$. Since $a_n$ is outer in $A$, we may select a
 net $\{g_m\}\subset A$ 
so that for a fixed $n$, $a_n g_m \to 1$ weak*. If $p = \infty$
then $k_n \in M^{-1}$, hence $a_n  \in M^{-1}$. Therefore if in this case we left-multiply the net  
$\{a_ng_m\}_m$ by $a_n^{-1}$, we see that $g_m  \to a_n^{-1}$ weak*, so that $a_n^{-1} \in A$.  For other values of $p$, the fact that $k_n^{-1}\in L^p(M)$ combined with the equality $|a_n^{-1}|=|a_n^*|^{-1}=k_n^{-1}$, ensures that $a_n^{-1}\in L^p(M)$. Since $a_n g_m \to 1$ weak*, left-multiplying by 
$a_n^{-1}\in L^p(M)$ shows that 
$g_m \to a_n^{-1}$ in the weak topology of $L^p(M)$. To see this note that if $r \in L^q(M)$ ($q$ conjugate to $p$), then
$$\tau(g_m r) = \tau(a_n g_m (r a_n^{-1})) \to \tau(r a_n^{-1}) = 
\tau(a_n^{-1} r)$$
since $r a_n^{-1} \in L^1(M)$.  Hence by Mazur's theorem, $a_n^{-1}$
is in the norm closure of $A$ in $L^p(M)$.  In other words $a_n^{-1} \in H^p$.  

We may assume that $\Phi(ha_n) \geq 0$ for each $n$, by replacing each $a_n$ with $a_n v_n^*$ for a unitary $v_n$ in ${\mathcal D}$ with $\Phi(ha_n) v_n^* = \Phi(ha_n v_n^*)\geq 0$.  Note that we then still have that $|(a_n v_n^*)^*| = |a_n^*|$.  

To show that $h a_n \to \I$, we first prove  that $\Delta(h a_n) \to 1$.
  Note that $c_n \to 1$ pointwise on $(0,\infty)$, 
and so by the Borel functional calculus \cite[Theorem 5.6.26]{KR}
 $|h| k_n$ converges WOT to the spectral projection for $(0,\infty)$ of $|h|$. 
But as was noted earlier, the latter equals $\I$. 
We have $$\Delta(h a_n) = \Delta(|h| |(a_n)^*|)   
=  \Delta(|h| k_n) = \Delta(c_n(|h|)).$$
Let $\mu_{|h|}$ be as above. Then $$\int_0^\infty \, \log \, c_n(t) \, d \mu_{|h|}(t)
= \int_0^{1/n}  \, \log \, (tn)  \, d \mu_{|h|}(t),$$
and similarly $\int_0^\infty \, |\log \, c_n(t)| \, d \mu_{|h|}(t)
= -
\int_0^{1/n}  \, \log \, (tn)  \, d \mu_{|h|}(t)$ which in turn equals
$$ 
- \int_0^{1/n}  \,\left( \log \, t\, d \mu_{|h|}(t) - (\log n)\right) \,  \mu_{|h|}([0,\frac{1}{n}]) \leq - \int_0^{1/n}  \, \log \, t\, d \mu_{|h|}(t),$$
These are clearly uniformly bounded in $n$ by $- \int_0^{1}  \, \log \, t\, d \mu_{|h|}(t)$, which is finite since $\Delta(|h|) > 0$. 
By equation (\ref{specdis})
 above,  $\int_0^\infty \, f \, d \mu_{|h|}(t) =
\tau(f(|h|))$ for $f(t) = \log (c_n(t) + \epsilon)$, since there 
is an increasing sequence of Borel simple functions converging to $f$.
Thus $$\Delta(c_n(|h|)
+ \epsilon 1) = \exp \int_0^\infty \, \log \, (c_n(t) + \epsilon) \, d \mu_{|h|}(t).$$
Taking the limit as $\epsilon \to 0^+$ we have 
$$\Delta(h a_n) = \Delta(c_n(|h|)) = \exp \int_0^\infty \, \log \, c_n(t) \, d \mu_{|h|}(t) 
\to e^0 = 1$$ 
as $n \to \infty$, by Lebesgue's monotone convergence theorem again.

Next notice that $$\Delta(h a_n) = \Delta(\Phi(h) \Phi(a_n)) \leq
\tau(\Phi(h)\Phi(a_n)),$$
since $\Delta \leq \tau$ on positive elements \cite[Lemma 4.3.6]{AIOA}.  Thus 
$$\Delta(h a_n) \leq \tau(\Phi(h) \Phi(a_n)) = \tau(h a_n) \leq 1,$$ since $\|ha_n\|_\infty \leq 1$.  By squeezing, $\tau(h a_n)  \to 1$.
Then  $$\Vert 1 - h a_n \Vert_2^2 =  1 + \Vert  h a_n \Vert_2^2
-2 {\rm  Re}(\tau(h a_n)) \leq 2 (1- {\rm  Re}(\tau(h a_n))) \to 0 .$$
The equivalent modes of convergence now follow by Lemma \ref{gh}.

If $p < \infty$ then for any $r< p$ we have 
$$\Vert a_n^{-1} - h \Vert_r = \Vert (h a_n - 1) a_n^{-1}  \Vert_r \leq \Vert h a_n - 1 \Vert_s \,
\Vert a_n^{-1} \Vert_p ,$$
where $\frac{1}{p} + \frac{1}{s} = \frac{1}{r}$.  By 
\cite[Theorem 3.6]{FK}  we have $\Vert h a_n - 1 \Vert_s \to 0$, giving 
$\Vert a_n^{-1} - h \Vert_r \to 0$.  Thus $a_n^{-1} - h \to 0$ in measure 
by the equivalence of (i) and (ii) of \cite[Theorem 3.7]{FK}. Since $|a_n^{-1}| = k_n^{-1}  \leq  |h|  +1$, 
we deduce that $\Vert a_n^{-1} - h \Vert_p \to 0$ by  \cite[Theorem 3.6]{FK}.  

In the case $p = \infty$, we have that $\{a_n^{-1}\}$ is bounded in operator norm.
Then $$a_n^{-1} = 
(1 - h a_n) a_n^{-1} + h \to h$$ 
in the WOT (since the product of a sequence that converges
$\sigma$-strong* to $0$, with a bounded sequence, certainly converges to $0$
in WOT).    Thus $a_n^{-1} \to h$ weak*.   \end{proof}

{\bf Remark.}  In the proof of the above theorem the sequence $\{|h|-k_n^{-1}\}$ actually converges uniformly to 0, 
by the functional calculus applied to 
$t - b_n^{-1}(t)$.   In view of this it is
very tempting to try and modify the  construction above to show that $h-a_n^{-1}$ also converges
uniformly to 0.   
Certainly we have $\Vert a_n^{-1} \Vert \to \Vert h \Vert$ in the case $p = \infty$ (these are operator norms).   
Possibly it is also the case that every uniform outer is the norm limit of a sequence of strongly
outer elements. 

\medskip
 
  Note that in the case that $A = M = D$, uniform outer
is also the same as being outer:

\begin{corollary} \label{bbb}   Let $1\leq p \leq \infty$.  If $h \in L^p(M)$ 
and $h$ is outer (that is, $[h M]_p = L^p(M)$), then $h$ is uniform outer
in $L^p(M)$.  \end{corollary}   \begin{proof}  Following the previous proof we obtain $k_n \in M$ with $|h| k_n \to 1$ weak*, etc.  If $u$ is the unitary in the polar decomposition of $h$ then 
$h (k_n u^*) = u |h| k_n u^* \to u u^* = 1$.  Clearly $\Vert h (k_n u^*) \Vert \leq 1$ for each $n$.  \end{proof}


\begin{corollary}
Let $1\leq p\leq q\leq \infty$ with $h\in H^q$. Then $h$ is uniform outer in $H^q$ if and only if $h$ is uniform outer in $H^p$.
\end{corollary}

\begin{proof} Immediate from Lemma \ref{gh}, since $ha_n \to 1$ in $p$ or $q$
norm if and only if $ha_n \to 1$ in measure.
\end{proof}

\begin{proposition} \label{onetwo}  Let $1\leq p, q, r \leq \infty$ be given with  $\frac{1}{p}=\frac{1}{q}+\frac{1}{r}$ and $h\in H^p$ with $h=h_1h_2$ where $h_1\in H^q$ and $h_2 \in H^r$.  Then $h$ is right outer in $H^p$ if  both $h_1$ and $h_2$ are right outer in
$H^q$ and $H^r$ respectively.  Conversely, if 
$h$ is right outer (resp. left outer) in $H^p$ then $h_1$ is right outer in $H^q$ (resp. $h_2$ is left outer in $H^r$). \end{proposition}

\begin{proof}The first part follows by direct computation. Specifically given a sequence $\{a_n\}\subset A$ so that $h_1a_n\to \I$ in $H^q$, then since $p \leq q$, we surely have that $h_1a_n\to \I$ in $H^p$. It therefore remains to show that $\{h_1a_n\}\subset [hA]_p$. To see that this is indeed the case, note that for each $n$ we may select a sequence $\{b_m\}\subset A$ so that $h_2b_m\to a_n$ in $H^r$. Then as required,
 $hb_m = h_1 h_2b_m\to h_1 a_n$ in $H^p$.

Conversely suppose that $h$ is right outer. Let $\epsilon$ be given and select $a, b \in A$ with $\|ha-\I\|_p<\epsilon/2$ and $\|h_2-b\|_r<\epsilon/(2\|h_1\|_q\|a\|_\infty)$. (Here we used the outerness of $h$ and the fact that $h_2\in [A]_r$.) Then $$\|h_1(ba)-\I\|_p \leq \|h_1(ba)-ha\|_p+\|ha-\I\|_p<\epsilon.$$Hence $h_1$ is right outer in $H^p$. But then $h_1$ is also right outer 
in $H^q$. The alternative claim follows analogously.
\end{proof}

We showed in \cite{BL6} that if $\Delta(h) > 0$ then $h$ is left outer in $H^p$ 
if and only if 
$h$ is right outer in $H^p$.  We do not know if this is true in general  if 
$\Delta(h) = 0$.
  In the next section we explore a restricted setting where this does hold.  In view of the fact that $ab = 1$ 
if and only if $b a = 1$ in $M$, we expect that 
this is likely to hold for all outers, but cannot prove this as yet (see Corollary \ref{lr} for a partial result in this direction).  

\section{Diagonally commuting outers and approximation with invertibles}  

In this section we consider a hypothesis which implies many of the  very strong results hitherto only known for strongly outers \cite{BL6}.    For example, under this hypothesis
left and right outers are the same, and  every outer is  uniformly outer.

\begin{definition} \label{dc}  We say that an  outer $h\in H^p$ is {\em diagonally commuting} if there exists a unitary 
$u$ in $\mathcal{D}$ such that 
$|\Phi(h)|$ commutes with $u^*h$.    
\end{definition}

This is of course equivalent to saying that the spectral projections of $|\Phi(h)|$ commute with $u^*h$, where $u$ is as above. 
 Indeed it is well known for $S, T \in L^p(M)$ that 
if $T \geq 0$ then $ST = TS$  iff $S$ commutes with the spectral projections of $T$  (one 
direction of this follows easily from the spectral resolution of $T$, using basic principles from e.g.\ the first pages of \cite{Terp},  the fact that 
$\bar{M}$ is a Hausdorff topological algebra, and that for an element of $M$ commutation with $S$ 
in the latter algebra coincides with the usual commutation  in the theory of unbounded operators.  For the other direction, if $T$ commutes with $S$ (and hence $S^*$), then the Cayley transform of $T$, and the von Neumann algebra it generates, is in the von Neumann algebra of bounded operators commuting with $S$ and $S^*$.  Hence the latter contains the spectral projections of $T$).
We  are particularly interested in the case that $u = 1$, or $u$ is the unitary in the polar decomposition $\Phi(h)=u|\Phi(h)|$.

Definition \ref{dc} holds automatically in many cases of interest.  For example, it is obvious that:

\begin{lemma} \label{allis}  If $M=A=\mathcal{D}$, or if $\mathcal{D}$ is  a subset of the center of  $M$, then every outer in 
$H^p$ is diagonally commuting.  \end{lemma}

\begin{proof}  If  $M=A=\mathcal{D}$ take $u$ to be the unitary in the polar decomposition $\Phi(h)=u|\Phi(h)|$. The other case is trivial.
 \end{proof}

Before proceeding, we first make the following observation:  Suppose  that $h \in H^p$ is outer, and $e$ is a projection in ${\mathcal D}$ commuting with $h$, and 
let $h_e = he + \kappa e^\perp$, for a scalar $\kappa > 0$.  Then $h_e A$ contains
$\frac{1}{\kappa}h_e e^\perp=e^\perp$, and $he$.  Thus $[h_e A]_p$ contains $[ehA]_p = e[hA]_p = e H^p$, and in particular 
contains $e$, and therefore also $\I = e + e^\perp$.  So $h_e$ is outer in $H^p$.  Morover $\Phi(h_e) = 
\Phi(h) e + \kappa e^\perp$, which is easy to understand.    In particular, $|\Phi(h_e)|^2 = |\Phi(h)|^2 e + \kappa^2 e^\perp$.

Now let   $e_n$ be the spectral projections of $|\Phi(h)|$ 
corresponding to the interval $(1/n,\infty)$, and we will be supposing that these commute with $u^*h$ where $u$ is a fixed unitary 
 in $\mathcal{D}$.

\begin{lemma} \label{dlem0}
Let $h$ be a diagonally commuting outer in $H^p$, where $1 \leq p \leq \infty$, and let $u$ be the associated  unitary  in $\mathcal{D}$.   Then $h_n = he_n+\tfrac{1}{n} u(\I-e_n)$ is
a sequence of strongly (hence uniform) outer elements in $H^p$ which converges
to $h$ in $p$-norm if $p < \infty$, and converges weak* (even  $\sigma$-strong*) if $p = \infty$.
 \end{lemma}

\begin{proof}    
By the observation above the Lemma,
the outerness of $u^*h$ ensures that $u^*he_n+\tfrac{1}{n}(\I-e_n)$ is outer, and therefore so also is $h_n = he_n+\tfrac{1}{n}u(\I-e_n)$. The $h_n$'s are even strongly outer 
because $u^* \Phi(h_n) = 
u^* \Phi(h) e_n +  \tfrac{1}{n} \, e^\perp$, so that $|\Phi(h_n) |^2 = |\Phi(h)|^2 e_n + \tfrac{1}{n^2} \, e^\perp \geq \tfrac{1}{n^2} \I$,
whence $\Delta(h_n) = \Delta(\Phi(h_n)) > 0$.   

 Since $\tfrac{1}{n}(\I-e_n)$ converges uniformly to 0, and $h e_n \to h$ $\sigma$-strong*, this proves the claim regarding convergence for the case $p=\infty$. In the case $p<\infty$ note that  
$h e_n \to h$ in measure, since $e_n \to 1$ in measure (by Lemma \ref{gh} if necessary). 
 By  \cite[Theorem 3.6]{FK} applied to the adjoints, 
$$\Vert h e_n - h \Vert_p = \Vert (h e_n - h )^* \Vert_p = \Vert e_n h^*- h^* \Vert_p \to 0$$
(note that $|e_n h^*| \leq |h^*| \in L^p(M)$).   \end{proof}

\begin{theorem}
\label{last} Let $h\in H^p$ be a diagonally commuting outer, and let $u$ be the associated  unitary  in $\mathcal{D}$. Then $h$ is 
 right  outer in $H^p$  if and only if  there is a sequence
$\{ e_n \}$ of projections in ${\mathcal D}$ increasing to $\I$, such that $u^* h e_n$ is strongly outer in $H^p(e_n M e_n)$.
In this case,   $h$ is uniformly outer in $H^p$.   
\end{theorem}

\begin{proof}  
The major portion of one direction of the proof is contained in 
the previous proof. We merely verify that $u^* he_n$ is then actually strongly outer in $H^p(e_nMe_n)$. Notice that on selecting a net $\{g_m\}$ in $A$ for which $u^* h_ng_m\to \I$, compressing this expression with $e_n$ will yield the conclusion that $u^* he_n(e_ng_me_n)\to e_n$, and hence that $u^* he_n$ is outer in $H^p(e_nMe_n)$. To see that it is strongly outer, one uses the fact that $|\Phi(h)|^2 e_n\geq\frac{1}{n^2}e_n$. If $\Delta_n$ is the Fuglede-Kadison determinant for $e_nMe_n$ computed using the normalised trace $\frac{1}{\tau(e_n)}\tau(\cdot)$, this shows that
 $\Delta_n(u^* he_n)\geq\Delta_n(u^* \Phi(h)e_n) = \Delta_n(|\Phi(h)| e_n) \geq \frac{1}{n}$. 

For the other direction, first suppose that $p = \infty$, and note that    
by Theorem \ref{fak}, there is a sequence 
$\{ a^n_m \}_m$ of outers in $e_nAe_n$ with $\Vert h e_n a^n_m \Vert \leq 1$ and $u^* h e_n a^n_m \to e_n$ weak*.   Then we have $\I = \lim_n \, \lim_m u^* h e_n a^n_m$, a double weak* limit of contractions.  
This may be rewritten  
as a net  of contractions converging weakly in $H^r$ for any $r \in  
[1,\infty)$.  By Mazur's theorem applied to the set of operator norm
contractions in $u^* hA$ inside $H^r$, 
taking convex combinations yields a sequence $\{ b_n \} \subset A$ with $\Vert h b_n \Vert \leq  
1$ and  $u^* h b_n \to 1$ in $r$-norm.   If $p = \infty$
the latter implies $u^* h b_n \to 1$ weak*.  Thus $u^* h$, and hence $h$, is uniformly outer in $H^p$.
  \end{proof}  

\begin{corollary}  Let $h\in H^p$ ($1\leq p\leq \infty$) be a diagonally commuting right outer.  Then $h$ is also left  outer in $H^p$.
\end{corollary}

\begin{proof}    Since multiplying by a unitary in ${\mathcal D}$ does not effect being outer, we may assume that $u = 1$.  Then 
this follows by the symmetry of the clause after the `if and only if'  in Theorem \ref{last}.  
\end{proof}

\begin{corollary}  Let $h\in H^p$ be a diagonally commuting outer.  If $p < \infty$ then there exists a sequence $(z_n)$ of outer elements in $A$,
with $z_n^{-1} \in H^p$, such that $z_n^{-1} \to h$   in $p$-norm.
If $h \in  A$, then $h$ is a weak* limit (which is also a limit in $r$-norm for any $1\leq r<\infty$) of a bounded sequence of invertible elements 
$z_n \in A^{-1}$.  \end{corollary}     

\begin{proof}  
Simply apply the proof of Theorem 2.8 to the  $h_n$ 
defined in Lemma \ref{dlem0}:  one obtains a sequence of outers $a_m = a^n_m \in A$ with $a_m^{-1} \in H^p$
and
$\Vert a_m^{-1} - h_n \Vert_p \to 0$ as $m\to \infty$.  Choose one of these, say $a(n)^{-1}$ with
$\Vert  a(n)^{-1} - h_n \Vert_p < 1/n$.  Then $a(n)^{-1} \to h$ in $p$ norm, since $$\Vert a(n)^{-1} - h \Vert_p \leq \Vert a(n)^{-1} - h_n  \Vert_p
+ \Vert h_n - h \Vert_p \to 0.$$

In the case $p = \infty$  note that $h_n \to h$ in measure.  Also  in  
the proof of Theorem \ref{fak}
we have $a_m^{-1} - h_n = (1 - h_n a_m) a_m^{-1}$, and since 
$\{a_m^{-1}\}$ is uniformly bounded, we see that $a_m^{-1} \to h_n$
in any $p$-norm, and hence in measure.
Thus one has a double limit in measure $\lim_n \, \lim_m  (a^m_n)^{-1} =
\lim_n h_n = h$, with $\Vert (a^m_n)^{-1} \Vert \leq
\Vert h_n \Vert + \frac{1}{n} \leq
\Vert h \Vert + \frac{2}{n}$ always.   Since the convergence in  
measure topology is metrizable,
there is a bounded sequence of invertible elements
$z_n \in A^{-1}$ converging in measure to $h$. (To see that the topology of convergence in measure on $\widetilde{M}$ is metrizable, note that by \cite[Theorem I.28]{Terp} and its proof, $\widetilde{M}$ is a complete Hausdorff topological $*$-algebra under this topology, with a countable base of neighbourhoods at 0.)    By \cite[Theorem  
3.6]{FK}, the convergence
is also in $2$-norm, which implies weak* convergence since it is a  
bounded sequence.
   \end{proof}   

We conjecture that in the general case refinements of the above ideas will essentially still work to give similar results,
but suspect that  this will need a quite  sophisticated modification of the proof of Theorem \ref{fak}.

\section{Inner-outer factorization and the characterization of
outers}
 
\begin{lemma} \label{dlem}
Let $h\in H^p$ be given, where $1 \leq p < \infty$.  If  $\Phi(h)$ is outer in $L^p(\mathcal{D})$
then $h$ is of the form $h=ug$ where $g\in H^p$ is outer and $u\in A$ is a unitary.
If $\Phi(h)$ is strongly outer then so is $g$.
 \end{lemma}

\begin{proof}  Suppose that $\Phi(h)$ is outer in $L^p(\mathcal{D})$.  
Consider the map $\Phi_{\vert hA} : hA \to \Phi(h) \mathcal{D}$.  This
is a $\mathcal{D}$-module map.  Its  kernel is $hA_0$ since $$\Phi(ha) =
\Phi(h) \Phi(a) = 0 \; \Rightarrow \; \Phi(a) = 0  \; \Rightarrow \; a \in A_0 .$$
Thus $hA/(hA_0) \cong \Phi(h) \mathcal{D}$ as $\mathcal{D}$-modules.  Since $\Phi(h) \mathcal{D}$ is
singly generated by $\Phi(h)$ as a $\mathcal{D}$-module, the module
$hA/(hA_0)$ is singly generated by $h + (hA_0)$.  It is easy to see that this
implies that $[hA]_p/[hA_0]_p$ is topologically singly generated by $h + [hA_0]_p$.
That is, the wandering quotient has a cyclic vector for the action of $\mathcal{D}$.

Next, if $hd \in [hA_0]_p \subset [A_0]_p = {\rm Ker}(\Phi_{\vert H^p})$, then $\Phi(h) d
= \Phi(hd) = 0$.  So $d = 0$ since $\Phi(h)$ is outer.  This implies that
$h + [hA_0]_p$ is a separating vector for the action of $\mathcal{D}$ on the wandering quotient.

We have shown that if $\Phi(h)$ is outer in $L^p(\mathcal{D})$, then
$h + [hA_0]_p$ is a cyclic separating vector for the wandering quotient
(in the case
$p = 2$ we can use the wandering subspace, and the cyclic separating vector will be $P(h)$, 
where in this case $P$ is the orthogonal projection  $P:[hA]_2\to[hA]_2 \ominus [hA_0]_2$.)   
As in \cite{BL3} this implies that $h=ug$ where $g\in H^p$ is outer and $u\in M$
is unitary (see \cite[Proposition 4.7]{BL3} (ii) and the lines before the Closing Remark of that paper).
 It remains to show that $u\in A$. To see this notice that for any $a_0\in A_0$ we have that $0=\tau(ha_0)=\tau(u(ga_0))$. Since $g$ is outer, we also have that $[gA_0]_p=H^p_0$. Hence we in fact have that $0=\tau(ua_0)$ for any $a_0\in H^p_0$, and in particular for all $a_0\in A_0$. But then $u \in H^2=L^2\ominus (H^2_0)^*$. Thus as required $u\in A= H^2\cap M$.

For the strongly outer statement, if $0 < \Delta(\Phi(h)) = \Delta(\Phi(u)) \Delta(\Phi(g))$ then $0 < \Delta(\Phi(g)) = \Delta(g)$.  
\end{proof}

{\bf Remark.}  Thus for 
$h\in H^p$ as above, $h$ is outer in $L^p(M)$ if $\Phi(h)$ is outer in $L^p(\mathcal{D})$.

\begin{corollary} \label{lr}   If $h \in H^p$ is left outer then $h = u k$ for a unitary $u \in A$ and 
a left outer $k \in H^p$ which is also right outer.  
\end{corollary}   

\begin{proof}  If $h$  is left outer then  $\Phi(h)$ is outer  in $L^p(\mathcal{D})$.
By Lemma \ref{dlem} and Proposition \ref{onetwo} we have $h = uk$, with $u$ inner and $k$ both left and right outer. \end{proof} 


The following must be well known:

\begin{lemma} \label{mstk}  Let $h\in L^p(M)$ be given, where $1 \leq p < \infty$,  and suppose that 
$\Vert ah  \Vert_p = \Vert h \Vert_p$ for a contraction $a \in M$.  Then $h = a^* a h$.  If in addition the left support of $h$ is $\I$ then $a$ is a unitary.   \end{lemma} \begin{proof} 
We have $\tau((h^* a^* a h)^{\frac{p}{2}} )= \tau((h^* h)^{\frac{p}{2}})$.  Let $T = h^* h, S = h^* a^* a h$. Then 
$S \leq T$.  By \cite[Lemma 2.5]{FK} (ii) we have $\mu_s(S)
\leq \mu_s(T)$ for all $s > 0$.  Let $r = \frac{p}{2} \geq \frac{1}{2}$.
By \cite[Lemma 2.5]{FK} (iv) we have $\mu_s(S^r) = \mu_s(S)^r \leq \mu_s(T)^r =
\mu_s(T)^r$ for all $s > 0$.
Thus $$0 = \Vert T \Vert^r_r - \Vert S \Vert_r^r = \int_0^\infty \, (\mu_s(T)^r -
\mu_s(S)^r) \, ds .$$
Thus $\mu_s(T)^r = \mu_s(S)^r$, or $\mu_s(T) = \mu_s(S)$, a.e.  Now $S^{\frac{1}{2}} \leq T^{\frac{1}{2}}$
by \cite[Lemma 2.3]{Schmitt}, and 
$$\tau(T^{\frac{1}{2}}-S^{\frac{1}{2}})  = \int_0^\infty \, (\mu_s(T)^{\frac{1}{2}} - \mu_s(S)^{\frac{1}{2}}) \, ds = 0,$$ so that $S = T$.  
Thus $\Vert (1-a^* a)^{\frac{1}{2}} h \Vert_p^2 = \Vert h^* (1-a^* a) h \Vert_r= 0$, so that 
$(1-a^* a)h =(1-a^* a)^{\frac{1}{2}}[(1-a^* a)^{\frac{1}{2}} h] = 0$ and $h = a^* a h$.  The last statement is obviously true
since $M$ is finite.  \end{proof}

\begin{theorem}\label{Houterp} 
Let $h\in H^p$ be given, where $1 \leq p < \infty$,  and let $P$ be the canonical quotient map from $[hA]_p$ to $[hA]_p/[hA_0]_p$ (if $p =2$ this is the orthogonal projection $P:[hA]_2\to[hA]_2\ominus[hA_0]_2$). Then $h$ will be outer if and only if $\Phi(h)$ is outer in $L^p(\mathcal{D})$ and $\|\Phi(h) \|_p=\|P(h)\|$.
In this case, and if  $p = 2$, then in fact $\Phi_{\vert [hA]_2} = P$.  \end{theorem}

\begin{proof} 
First suppose that $h$ is outer. To see that $\Phi(h)$ is outer, select a sequence $\{a_n\}\subset A$ so that $ha_n \to \I$ in $L^p$-norm. On applying $\Phi$ we see that $\Phi(h)\Phi(a_n)\to \I$, which is sufficient to ensure that $\Phi(h)$ is outer in $L^p(\mathcal{D})$. 
For any $a_0 \in [A_0]_p$, we clearly have that $$\|\Phi(h)\|_p=\inf_{a_0\in A_0}\|h+a_0\|_p.$$(To see this recall that $\Phi$ is 
a contractive projection from $H^p$ onto $L^p(\mathcal{D})$ with kernel $H^p_0$.) Since $h$ is outer we have that $[hA_0]_p = [A_0]_p$ (note that if $h a_n \to 1$ and $b \in A_0$ then $h (a_nb) \to b$, so that $A_0 \subset [hA_0]_p \subset [A_0]_p$, and so
 $[A_0]_p = [hA_0]_p$). If we apply this fact to the above equality, it is clear that in fact
$$\|\Phi(h)\|_p= \inf_{a_0\in A_0}\|h+a_0\|_p =\inf_{a_0\in A_0}\|h+ha_0\|_p=\|P(h)\|.$$  

Conversely suppose that $\Phi(h)$ is outer in $L^p(\mathcal{D})$ and $\|\Phi(h)\|_p=\|P(h) \|_p$.
By Lemma \ref{dlem}, $h=ug$ where $g\in H^p$ is outer and $u\in A$ is a unitary.  
In fact one may prove that $u \in \mathcal{D}$ by a tweak of the argument at the end of the proof of \cite[Proposition 4.8]{BL6}. 
To this end observe that $$\Vert \Phi(h) \Vert_p = \Vert \Phi(u) \Phi(g) \Vert_p \leq \Vert \Phi(g) \Vert_p  = \inf_{a_0\in A_0}\| g + ga_0\|_p=
\Vert P(h) \Vert  = \Vert \Phi(h) \Vert_p,$$
where the middle equality follows from the last paragraph, and the  equality  after that uses the fact that 
$\| g + ga_0\|_p=  \| ug + uga_0\|_p= \| h + ha_0\|_p$.     We conclude that 
 $\Vert \Phi(u) \Phi(g) \Vert_p = \Vert \Phi(g) \Vert_p$, or equivalently by Lemma \ref{mstk} that $\Phi(u)$ is unitary, in 
which case $\Phi((u - \Phi(u))^* (u - \Phi(u)) = 0$. So $u = \Phi(u) \in \mathcal{D}$, and $h$ is outer.
\end{proof}

{\bf Remarks.}  1) \ The last result is a variant of \cite[Proposition 4.8]{BL6},
which says that $h\in H^2$ is outer if and only if $\|\Phi(h) \|_2=\|P(h)\|_2$
and there is a cyclic separating vector for the right action of $\mathcal{D}$ on the 
wandering subspace $[hA]_2 \ominus [hA_0]_2$.  A perusal of the proofs of these 
two results, reveals that \cite[Proposition 4.8]{BL6} is a special case of Theorem \ref{Houterp}.     

2) \ Suppose that $\square : L^p(M)  \to [0,\infty)$ is any function such that
$\square(h) = \|\Phi(h) \|_p$ if $h \in H^2$ is outer, and $\square(f) = \square(|f|)$ 
if $f \in L^p(M)$.
  In the last proof (see also \cite[Proposition 4.8]{BL6}) one sees that $P$ gives rise to
such a function 
(It seems plausible that in for example the case $p = 1$, 
the determinant like quantity
$\tau(\exp(\Phi(\log |a|)))$ is also such 
a function, but we have not thought about this.)   Then $h\in H^p$ is outer if and only if $\Phi(h)$ 
is outer in $L^p(\mathcal{D})$ and $\|\Phi(h) \|_p = \square(h)$.  
The proof is essentially identical to that of Theorem \ref{Houterp}.
Indeed if $h=ug$ where $g\in H^p$ is outer and $u\in A$ is a unitary, 
then $\square(h) = \square(g) = \|\Phi(g)  \|_p$.
If $\square(h) = \|\Phi(h) \|_p = \|\Phi(u) \Phi(g)  \|_p$, 
the proof of Theorem \ref{Houterp} then shows that $h$ is outer.

3) \ In the case $p =2$ there is a  quick proof of Theorem \ref{Houterp}.
  By the hypothesis we may select a sequence $\{a_n\}\subset A_0$ so that $\|h(\I+a_n)\|_2 \to \|\Phi(h)\|_2$ as $n\to\infty$. Then
\begin{eqnarray*}
\|\Phi(h)-h(\I+a_n)\|^2&=&  \tau(|\Phi(h)-h(\I+a_n)|^2)\\
&=& \tau(|\Phi(h)|^2- \Phi(h)^*h(\I+a_n)-(\I+a_n^*)h^*\Phi(h))\\
&& + \tau(|h(1+a_n)|^2)\\
&=& \tau(\Phi(|\Phi(h)|^2- \Phi(h)^*h(\I+a_n)-(\I+a_n^*)h^*\Phi(h)))\\
&& + \tau(|h(1+a_n)|^2)\\
&=& -\tau(|\Phi(h)|^2) + \|h(1+a_n)\|_2^2\\
&\to& 0
\end{eqnarray*}
as $n\to \infty$. It follows that $\Phi(h)-h\in [hA_0]_2$ (or $\Phi(h) \in [hA]_2$), and hence that $h$ is outer by Proposition \ref{innam}. 

\bigskip

The last result proves Theorem \ref{Houterp0}, since the notation 
$\delta^1(h)$ in that result is just $\|P(h)\|$.  We remark that in the light of Theorem \ref{Louter0}
one might expect expressions $\delta^e(h)$ to occur in Theorem \ref{Houterp0} for projections $e \in \mathcal{D}$,
and in fact one can do this but it is unnecessary--following as an automatic consequence
of the $\delta^1(h)$ expression.

\begin{lemma} \label{dlemi2}  If $1 \leq p < \infty$ and $h\in H^p$  then   $\Phi(h)$ is outer in
 $L^p(\mathcal{D})$
if and only if  $h$ is either outer or of the form $h=ug$ where $g\in H^p$ is outer and $u\in A$  is a unitary with $\Phi(u)$ outer and $\Vert \Phi(u) \Vert_p < 1$.  \end{lemma}

\begin{proof}  
 If $h=ug$ is as described, then $[\Phi(h) \mathcal{D}]_p = [\Phi(u) \Phi(g) 
\mathcal{D}]_p= [\Phi(u) \mathcal{D}]_p$,
since any $d \in \mathcal{D}$, and in particular $d = 1$, may be approximated by a sequence
$\Phi(g)d_n$ with $d_n \in \mathcal{D}$, and then $\Phi(u) \Phi(g)d_n \to \Phi(u) d$.  So 
$\Phi(h)$ is outer in $L^p(\mathcal{D})$ if and only if $\Phi(u)$ is outer there.
If $u$ itself is 
outer, then it is easy to see that $u \in \mathcal{D}$
(since $1 \in u [A]_p$ implies that $u^* \in [A]_p \cap M = A$).  Hence $ug = h$
is then outer in $H^p$.  If $u$ is not 
outer, then by Theorem \ref{Houterp} we must have
$\Vert \Phi(u) \Vert_p \neq 1$, and so $\Vert \Phi(u) \Vert_p < 1$.  \end{proof}

We remark that in \cite{BL3}  it was shown (see the lines before the Closing Remark of that paper) that if $f\in L^p(M)$ ($1 \leq p \leq \infty$), then $f$ is of the form $f=uh$ for some outer $h\in H^p$ and a unitary $u\in M$ whenever the right-wandering subspace of $[fA]_p$ (respectively right-wandering quotient of $[fA]_p$) has a nonzero separating and cyclic vector for the right action of $\mathcal{D}$.  Thus condition (3) below implies the validity of (1) for any $1\leq p\leq \infty$.   We now sharpen this statement: 

\begin{theorem}\label{Louter} Let $f\in L^p(M)$ ($1\leq p\leq \infty$).  The following are equivalent:
\begin{enumerate}
\item $f$ is of the form $f=uh$ for some outer $h\in H^p$ and a unitary $u\in M$;
\item the map $\mathcal{D} \to [fA]_p/[fA_0]_p: d\to P(fd)$ is injective, where $P$ is the quotient map $P:[fA]_p\to[fA]_p/[fA_0]_p$ (If $p=2$, $P$ may be taken to be the orthogonal projection $P:[fA]_2\to[fA]_2\ominus[fA_0]_2$.)
\item the right-wandering quotient of $[fA]_p$ (respectively right-wandering subspace of $[fA]_2$ if $p=2$) has a nonzero separating and cyclic vector for the right action of $\mathcal{D}$.
\item $fe \notin [fA_0]_p$ for every  nonzero projection $e$ in $\mathcal{D}$.
\end{enumerate}
\end{theorem}

\begin{proof} As was noted in the discussion preceding the Theorem, (3) implies (1).  Clearly (2) implies (4).

Suppose that (1) holds, that is let $f$ be of the form $f=uh$ for some outer 
$h\in H^p$ and a unitary $u\in M$.   Then $[fA]_p = u [hA]_p = u H^p$ and $[fA_0]_p = u [h A_0]_p
= u [A_0]_p$ (see the proof of Theorem \ref{Houterp}).
Thus the right-wandering quotient
$[fA]_p/[fA_0]_p
= (u [A]_p)/(u[A_0]_p) = u L^p(\mathcal{D})$, which ensures the validity of (2) and (3).
Thus (1)  implies both (2) and (3). 

It remains to prove that (4) implies (1). So  assume that (4) holds.

We first deal with the case $p=2$. Setting $e=\I$, we clearly have that $f\notin[fA_0]_2$. With $v$ denoting the orthogonal projection of 
$f$ onto $[fA_0]_2$, it then follows from \cite[Lemma 4.9]{BL6} that $|f-v|\in 
L^2(\mathcal{D})$. Since $v \in [fA_0]_2$, we clearly have that $va\in [fA_0]_2$ for any $a\in A$. 
Hence
\begin{equation}\label{eq:aa}(f-v)d +[fA_0]_2 = fd + [fA_0]_2 , \qquad d \in A .\end{equation}
Let  $e$ be a nonzero projection in $\mathcal{D}$.  By hypothesis 
$fe \notin [fA_0]_2$, so that $(f-v)e \neq 0$ by (\ref{eq:aa}). This implies that the support projection
$\mathrm{supp}(|f-v|) = \I$ (otherwise setting $e=\I-\mathrm{supp}(|f-v|)$ would contradict the statement just made). With $u$ denoting the partial isometry in the polar decomposition of $(f-v)$, this means that $\I=u^*u$. Since $M$ is finite, this can only be the case if $u$ is in fact a unitary. 

We proceed to show that $\tau(u^*fa_0)=0$ for all $a_0\in A_0$.
 To see this, let $e_n\in\mathcal{D}$ be the
spectral projection of $|f-v|$ corresponding to the interval
$[0,1/n]$.  By the Borel functional calculus, 
we firstly see that for each $n$ there exists $r_n \in {\mathcal D}$ so that $1-e_n = |f-v| r_n$.  
Secondly, $e_n$ converges strongly to $\I-\mathrm{supp}(|f-v|)=0$.
Given $a_0\in A_0$, this in turn ensures that
\begin{eqnarray*}
\tau(a_0^*f^*u)&=& \tau(a_0^* f^* ue_n) (\I) + \tau(a_0^* f^* u (1-e_n))\\
&=& \tau(a_0^* f^* ue_n) + \tau(a_0^* f^* (f-v) r_n) \\
&=& \tau(a_0^* f^* ue_n) \to 0.
\end{eqnarray*}
(In the last equality above, we made use of the fact that $fa_0r_n^* \in [fA_0]_2$ with $f-v$ orthogonal to
$[fA_0]_2$.)
As required it therefore follows that $$\tau(u^*fa_0)=0\quad\mbox{for all}\quad a_0\in A_0.$$
We  make two deductions from this fact. First,  $u^* f \in L^2(M) \ominus [A_0^*]_2 = H^2$. (So on setting $h=u^*f$, $f$ will be of the form $f=uh$ for a unitary $u\in M$ and some $h\in H^2$.) Second, since $v$ is a limit of terms of the form $fa_0$ (where $a_0\in A_0$), it will also follow from this fact that $\tau(u^*vd)=0$ for any $d \in \mathcal{D}$.
 For any $d\in \mathcal{D}$, we will then have that
\begin{eqnarray*}
\tau(\Phi(h)d) &=& \tau(\Phi(hd))\\
&=& \tau(hd)=\tau(u^*fd)\\
&=& \tau(u^*(f-v)d) = \tau(|f-v| d).
\end{eqnarray*}
Since $d\in \mathcal{D}$ was arbitrary and since $|f-v|\in L^2(\mathcal{D})$, we must have $\Phi(h)=|f-v|$. Notice that
$$[\Phi(h)\mathcal{D}]_2=[|f-v|\mathcal{D}]_2= [\mathrm{supp}(|f-v|)\mathcal{D}]_2=L^2(\mathcal{D}).$$Hence $\Phi(h)$ is outer in 
$L^2(\mathcal{D})$.  So by Lemma \ref{dlem},
$h = v k$, and $f = (uv) k$, 
for a unitary $v$ ($u$ as before),  and outer $k \in H^2$. Thus the implication is valid for the case $p=2$.  

Next let $f\in L^p$ where $1\leq p \leq \infty$. We first show that $f$ is then injective with dense range. This in turn will then enable us to factorise $f$ in a way which will enable us to apply the result for $p=2$ to obtain the desired conclusion. Let $f=w|f|$ be the polar decomposition of $f$. It is an exercise to see that for any projection $e$, $fe= (w|f|^{1/2})|f|^{1/2}e = fe \notin (w|f|^{1/2})[|f|^{1/2}A_0]_{2p} \subset [fA_0]_p$ whenever $|f|^{1/2}e \notin [|f|^{1/2}A_0]_{2p}$. Hence the validity of (4) for $f\in L^p(M)$, ensures that (4) also holds for $|f|^{1/2}$ considered as an element of 
$L^{2p}(M)$. Given a projection $e\in M$, an application of \cite[Lemma 4.2]{BL3} ensures that  $|f|^{1/2}e \in [|f|^{1/2}A_0]_{2p}$ if 
and only if $|f|^{1/2}e \in [|f|^{1/2}A_0]_{2}$. Thus (4) holds for $|f|^{1/2}$ considered as an element of $L^2(M)$. By what we have 
already proved, $|f|^{1/2}$ must then be of the form $|f|^{1/2} = ug$ for some unitary $u\in M$ and an outer $g\in H^2(M)$. But then $|f|^{1/2}$ must be injective with dense range by \cite[Lemma 4.2]{BL6}. Since $M$ is finite and the projections $\mathrm{supp}(|f|^{1/2})$ and $\mathrm{supp}(|f^*|^{1/2})$ equivalent, we also have that $\mathrm{supp}(|f^*|^{1/2})=\I$

We show that there exists a positive element $k$ of $M$ for which $k^{-1}$ is itself an element of $L^{2p}(M)$, $kf\in L^{2p}(M)$, and $\Delta(k^{-1}) > 0$. The existence of such a $k$ can be verified by modifying aspects of the proof of Theorem \ref{fak}. We outline the main features of this construction, and refer the reader to the proof of Theorem \ref{fak} for technical details. Let $b(t) = 1/t$ if $t > 1$, and  $b(t) = 1$ if $0 \leq t \leq 1$. (This corresponds to $b_1$ in the proof of Theorem \ref{fak}.) The multiplicative inverse of the function $b$ defined above is nothing but $\frac{1}{b}(t) = t$ if $t > 1$, and  $\frac{1}{b}(t) = 1$ if $0 \leq t \leq 1$. For this function we have $\frac{1}{b}(t) 
\leq 1+t$.

Now consider the operators $k = b(|f^*|^{1/2})\in M$ and $\frac{1}{b}(|f^*|^{1/2})$. By the Borel functional calculus we then have that $$k\frac{1}{b}(|f^*|^{1/2})=\frac{1}{b}(|f^*|^{1/2})k=\mathrm{supp}(|f^*|^{1/2})=\I.$$Hence the affiliated operator $\frac{1}{b}(|f^*|^{1/2})$ is the inverse of $k$. Since $\frac{1}{b}(t) \leq 1+t$, the Borel functional calculus also ensures that $$k^{-1} = \frac{1}{b}(|f^*|^{1/2})\leq \I+|f^*|^{1/2}.$$In view of the fact that $f\in L^{p}(M)$ ensures that $|f^*|^{1/2}\in L^{2p}(M)$, it follows that $k^{-1}\in L^{2p}(M)$.

As we saw in the proof of Theorem \ref{fak}, for an operator $k$ thus constructed, we have that $\|k|f^*|^{1/2}\|_\infty \leq 1$ and that $\Delta(k) > 0$. But then also $\Delta(k^{-1})=\Delta(k)^{-1}>0$.

Recall that $f=w|f|$. Then also $f^*=w^*|f^*|$. Since by construction $k|f^*|^{1/2}\in M$, it follows that $kf = (k|f^*|^{1/2})|f^*|^{1/2}w\in L^{2p}(M)$.

Let $e\in M$ be a projection. It is then not difficult to see that if $kfe \in [kfA_0]_{2p}$, then $fe=k^{-1}kfe \in k^{-1}[kfA_0]_{2p}\subset [fA_0]_{p}$. Thus the validity of (4) for $f\in L^p(M)$, ensures its validity for $kf\in L^{2p}(M)$. 
Then (4) also holds for $kf$ considered as an element of $L^2(M)$, by \cite[Lemma 4.2]{BL3}.
 From what we have already proved, $kf$ must then be of the form $kf = u_0g_0$ for some unitary $u_0\in M$ and an outer $g_0\in H^2(M)$. Since then $|kf| = |g_0|\in L^{2p}(M)$, $g_0$ is in fact in $H^{2p}(M)$. Notice that we then have $f = k^{-1}u_0g_0$ with $\Delta(k^{-1}u_0) = \Delta(k^{-1})\Delta(u_0)= \Delta(k^{-1})>0$. Thus by the noncommutative Riesz-Szeg\"o theorem (see \cite[Corollary 4.14]{BL6}), there exists a unitary $u_1\in M$ and a strongly outer element 
$g_1\in H^{2p}(M)$ with $k^{-1}u_0=u_1g_1$. But then $f = u_1g_1g_0$. Since $g_1g_0$ is an outer element of $H^p$, it follows that $f$ is of the form described in (1).
\end{proof}

To deduce Theorem \ref{Louter0}  from the last theorem, note that (iv) of that result holds 
if and only if
$\inf \{ \Vert f e - f a_0 \Vert_p  : a_0 \in A_0 \} > 0$ for each projection $e \in \mathcal{D}$.  However by orthogonality it is easy to argue that 
$$\Vert f e - f a_0 \Vert_p = \Vert (f  - f a_0 )e + f a_0 e^{\perp}  \Vert_p  \geq  \Vert (f  - f a_0 )e \Vert_p  = 
\tau(e)^p \,  \Vert f - f a_0 \Vert^e_p.$$  Since $A_0 \, e \subset A_0$, we see that 
the last infimum is a strictly positive multiple of $\delta^e(f) > 0$.  Thus (iv) of Theorem \ref{Louter} is equivalent to 
$\delta(f) > 0$.  If $f \in H^p$ then $u = u \cdot 1 \in u [h A]_p = [fA]_p \subset H^p$, so that $u \in H^p \cap M = A$. 
If $f\in L^p(M)_+$ then these conditions imply that $f = u h = |f| = (h^* u^* u h)^{\frac{1}{2}} = |h|$.
Conversely if $f = |h|$ for outer $h$ then $\delta(f) = \delta(|h|)= \delta(h)$ since $|hx| = ||h|x|$ for any $x$,
and $\delta(h)$ is strictly positive by the implication already proved.  
This completes the proof of  Theorem \ref{Louter0}.

\end{document}